\documentclass[12pt, reqno,fleqn]{amsart}

\usepackage{amsmath}
\usepackage{amssymb}
\usepackage{amsfonts}
\usepackage{graphicx}
\usepackage{amsthm}
\usepackage{enumerate}
\usepackage{lscape}
\usepackage{dsfont}
\usepackage{color}
\usepackage{mathtools}

\usepackage{setspace}
\newcommand{\R}{\mathds{R}}                   
\newcommand{\z}{\mathds{Z}}
\newcommand{\q}{\mathds{Q}}
\newcommand{\CP}{\mathds{C}\mathrm{P}}

\newcommand{\f}{\rightarrow}                  
\newcommand{\C}{\mathds{C}}            
\newcommand{\de}{\partial}          

\newcommand{\K}{K\"{a}hler}
\newcommand{\GW}{GW}

\newcommand{\M}{\mathcal M}
\newcommand{\tr}{\operatorname{tr}}

\newcommand{\FS}{{\operatorname{FS}}}

\newcommand{\rk}{\operatorname{rk}}

\newcommand{\lmb}{\lambda}
\newcommand{\di}{{\operatorname{d}}}
\newcommand{\mult}{{\operatorname{mult}}}

\newcommand{\ncpt}{{\Omega}}
\newcommand{\Cut}{{\operatorname{Cut}}}

\newcommand{\nax}[1]{\left\|{#1}\right\|_{\max}}

\newtheorem{theor}{Theorem}
\newtheorem{prop}[theor]{Proposition}

\newtheorem{lem}[theor]{Lemma}
\newtheorem{cor}[theor]{Corollary}

\newtheorem{remark}[theor]{Remark}

\begin{document}

\title[Symplectic capacities  of  Hermitian symmetric spaces]{Symplectic capacities  of  Hermitian symmetric spaces}

\author{Andrea Loi}
\address{(Andrea Loi) Dipartimento di Matematica \\
         Universit\`a di Cagliari (Italy)}
         \email{loi@unica.it}

\author{Roberto Mossa}
\address{(Roberto Mossa) Dipartimento di Matematica \\
         Universit\`a di Cagliari (Italy)}
         \email{roberto.mossa@gmail.com}

\author{Fabio Zuddas}
\address{(Fabio Zuddas) Dipartimento di Matematica e Informatica \\
          Via delle Scienze 206 \\
         Udine (Italy)}
\email{fabio.zuddas@uniud.it}

\thanks{
The first author was (partially) supported by ESF within the program �Contact and Symplectic Topology�}

\begin{abstract}
Inspired by the work of G. Lu \cite{LU06} on pseudo symplectic capacities we  obtain several results on the Gromov width and the Hofer--Zehnder 
capacity of Hermitian symmetric spaces of compact type.
Our results and  proofs extend those obtained  by Lu for complex Grassmannians to Hermitian symmetric spaces of compact type.
We also compute the Gromov width and the  Hofer--Zehnder capacity for Cartan domains and their products.
\end{abstract}
 
\maketitle

\tableofcontents

\section{Introduction}
Consider the open ball of radius $r$,
\begin{equation}\label{ball}
B^{2n}(r)=\{(x, y)\in\R^{2n}\  |\  \sum_{j=1}^n|x_j|^2+|y_j|^2<r^2 \}
\end{equation}
in standard symplectic space $(\R^{2n}, \omega_0)$, where $\omega_0=\sum_{j=1}^n dx_j\wedge dy_j$.
The Gromov width of a $2n$-dimensional symplectic manifold $(M, \omega)$, introduced in \cite{GROMOV85}, is defined as
\begin{equation}\label{gromovwidth}
c_G(M, \omega)= \sup \{\pi r^2 \ |\ B^{2n}(r)\    \mbox{symplectically embeds into}  \  (M, \omega)\}.
\end{equation}

By Darboux's theorem $c_G(M, \omega)$ is a positive number.
Computations and estimates of the Gromov width for various examples can be found in 
\cite{BIRAN97},  \cite{BIRAN99}, \cite{BIRAN01}, \cite{castro}, \cite{GROMOV85}, \cite{JIANG00}, \cite{GWgrass},  \cite{LAMCSC11},  \cite{LU06}, \cite{LuDingQjao}, \cite{MCDUFF91}, \cite{MCDUFF94},  \cite{SCHLENK05}, \cite{GWcoadjoint}.

Gromov's width is an example of  \emph{symplectic  capacity} introduced  in \cite{HOFERZEHNDER90} (see also \cite{HOFERZEHNDER94}).
A map $c$ from the class  ${\mathcal C} (2n)$ of all symplectic manifolds of dimension $2n$ to $[0, +\infty]$
is called a \emph{symplectic capacity} if it satisfies the following conditions:

({\bf monotonicity}) if there exists a symplectic embedding $(M_1, \omega_1)\rightarrow (M_2, \omega_2)$ then 
$c(M_1, \omega_1)\leq c(M_2, \omega_2)$; 

({\bf conformality}) $c(M, \lambda\omega)=|\lambda|c(M, \omega)$, for every $\lambda\in\R\setminus \{0\}$; 

({\bf nontriviality}) $c(B^{2n}(1), \omega_0)=\pi =c(Z^{2n}(1), \omega_0)$.

\noindent
Here $B^{2n}(1)$ and $Z^{2n}(1)$ are the open unit  ball  and the  open cylinder in the standard $(\R^{2n}, \omega_0)$, i.e.
\begin{equation}\label{Zcil}
Z^{2n}(r)=\{(x, y)\in\R^{2n} \ | \ x_1^2+y_1^2<r^2\}.
\end{equation}

Note that the monotonicity property implies that $c$ is a symplectic invariant.
The existence of a capacity is not a trivial matter. 
It is easily seen that the Gromov width is the smallest symplectic capacity, i.e.   $c_G(M, \omega)\leq c (M, \omega)$ for any capacity $c$.  
Note that the  nontriviality property for $c_G$ comes from  the celebrated 
\emph{Gromov's nonsqueezing theorem} stating that the existence of a symplectic  embedding of
$B^{2n}(r)$ into 
$Z^{2n}(R)$ implies $r\leq R$. Actually it is easily seen that the existence of any capacity implies 
Gromov's  nonsqueezing theorem. 
H. Hofer and E. Zehnder  \cite{HOFERZEHNDER90} prove the existence of a capacity, denoted by   $c_{HZ}$
which plays, for example, an important role in the study of Hofer geometry on the group of symplectomorphisms of a symplectic manifold
and in establishing the existence of closed characteristics on or near an energy surface.
However, it  is difficult to  compute or estimate  $c_{HZ}$ even  for closed symplectic manifolds.  So far the only examples are closed surfaces, for which $c_{HZ}$ has been computed as the area \cite{SIBURG93}, and  complex projective spaces  and their products. H. Hofer and C. Viterbo \cite{HOFERVITERBO92}  proved that 
$c_{HZ}(\CP^n, \omega_{FS})=\pi$. This has been extended  by G. Lu  to the product of projective spaces (see Theorem 1.21 in \cite{LU06} or   (\ref{cHZprodproj}) below).
Lu's ingenious idea was  that of defining and studying  the concept  of 
pseudo  symplectic capacity,  more flexible than that of symplectic capacity,  and its link with Gromov-Witten invariants (see Section \ref{sectionLu} below).
This allows him  to  obtain several   important results, e.g.  the Gromov width of    Grassmannians and  their products and a  lower bound  for the Hofer--Zehnder capacity  for the product of any closed symplectic manifold with a Grassmannian. 
One of the  aims of this paper is to extend Lu's results  to the case of  Hermitian symmetric spaces of compact type (denoted by  \emph{HSSCT} in the sequel). 
Moreover, we compute the Gromov width and Hofer--Zehnder capacity of  Cartan's domains and their products.
In the next section we provide  a  description of   our results and the ideas of their proofs.

\section{Statements of the main results}

The following three theorems describe our results  about the Gromov width and the Hofer-Zehnder capacity
of HSSCT.

\begin{theor}\label{main}
Let $(M, \omega_{FS})$ be an irreducible HSSCT endowed with the canonical symplectic (\K) form $\omega_{FS}$ 
normalized  so that $\omega_{FS} (A)=\pi$ for the generator  $A\in H_2(M, \z)$. Then 
\begin{equation}\label{cGHSSCT}
c_G (M, \omega_{FS})=\pi.
\end{equation} 
\end{theor}

\begin{theor}\label{main2}
Let $(M_i, \omega_{FS}^i)$, $i=1, \dots, r$, be  irreducible HSSCT of complex dimension $n_i$ endowed with the canonical symplectic (\K) forms $\omega_{FS}^i$ 
normalized  so that $\omega_{FS}^i (A_i)=\pi$ for the generator  $A_i\in H_2(M_i, \z)$, $i=1, \dots , r$. Then 
\begin{equation}\label{cGprodHSSCT}
c_G \left( M_1\times\dots\times   M_r,  \omega_{FS}^1\oplus\dots  \oplus \omega_{FS}^r\right)=\pi.
\end{equation}
Moreover, if  $a_1, \dots ,a_r$ are nonzero constants, then
\begin{equation}\label{cGupboundprodHSSCT}
c_G \left( M_1\times\dots\times   M_r, a_1 \omega_{FS}^1\oplus\dots  \oplus a_r\omega_{FS}^r\right)\leq \min \{|a_1|, \dots ,|a_r|\}\pi 
\end{equation}
and
\begin{equation}\label{cHZlowerboundprodHSSCT}
c_{HZ} \left( M_1\times\dots\times   M_r, a_1 \omega_{FS}^1\oplus\dots  \oplus a_r\omega_{FS}^r\right)\geq (|a_1|+ \dots +|a_r|)\pi .
\end{equation}
\end{theor}
\begin{theor}\label{main3}
Let $(M, \omega_{FS})$ be an irreducible HSSCT and  $(N, \omega)$ be any closed symplectic manifold. Then, for any nonzero real number $a$,
\begin{equation}\label{cGMprodHSSCT}
c_G(N\times M, \omega\oplus a\omega_{FS})\leq|a|\pi.
\end{equation}
\end{theor}
Formulas (\ref{cGHSSCT}) and  (\ref{cGprodHSSCT}) extend Theorem 1.15 and  formula   (22) in \cite{LU06} respectively  (valid for the Grassmannians) to  the case of HSSCT.
The   lower bounds  $c_G(M, \omega_{FS})\geq \pi$  in Theorem \ref{main} and
\[c_G \left( M_1\times\dots\times   M_r,  \omega_{FS}^1\oplus\dots  \oplus \omega_{FS}^r\right)\geq \pi\]
in  Theorem \ref{main2} are obtained by using the  results in \cite{DiScalaLoi08} 
which imply  the existence of a
symplectic embedding of   the noncompact dual $(\Omega, \omega_0)$ of  $(M, \omega_{FS})$ into  $(M, \omega_{FS})$
(where $\omega_0$ is the standard symplectic form of ${\Omega}\subset \C^n$, being $n$ the complex dimension of $M$) and by the existence of a  symplectic  embedding of $B^{2n}(1)$ into
$(\Omega, \omega_0)$  (see Sections \ref{hpjts}  and \ref{sectionembeddings} below for details).
The upper bounds   $c_G(M, \omega_{FS})\leq \pi$ and 
\[c_G \left( M_1\times\dots\times   M_r,  \omega_{FS}^1\oplus\dots  \oplus \omega_{FS}^r\right)\leq \pi\]
follow by the use of Lu's pseudo symplectic capacities and their estimation in terms of   Gromov-Witten invariants.
The key ingredient to obtain these upper bounds is the non vanishing of some 
 genus-zero three-points  Gromov-Witten invariants (cfr. Lemma \ref{gromovwittenHSSCT} in Section \ref{sectionproofs} below).  Inequality  (\ref{cGupboundprodHSSCT}),
 which extends (21) in \cite{LU06}  to HSSCT, is a consequence of (\ref{cGMprodHSSCT}) in Theorem \ref{main3}, which in turn extends \cite[Corollary 1.31]{LU06}.  

\vskip 0.3cm
When $M_j=\CP^1$ for all $j=1, \dots ,r$,
inequality (\ref{cGupboundprodHSSCT}) is indeed an equality,
i.e.
\begin{equation}\label{CP1SAMC}
c_G(\CP^1\times\cdots\times \CP^1, a_1\omega_{FS}\oplus\cdots\oplus a_r\omega_{FS})=\min\{|a_1|,\dots ,|a_r|\}\pi.
\end{equation}
 (see \cite[Example 12.5]{MCSA98} for a proof). 
 We do not know  the exact  value of 
\[c_G(\CP^{n_1}\times\cdots\times \CP^{n_r}, a_1\omega^1_{FS}\oplus\cdots\oplus a_r\omega^r_{FS})\]
if $n_i>1$ or $a_j\neq 1$ for some $i=1, \dots ,r$ or $j=1, \dots ,r$. 
 
\vskip 0.3cm
When $M$ and  the $M_j$'s  are projective spaces 
inequality 
 (\ref{cHZlowerboundprodHSSCT}) is   an equality, namely
 \begin{equation}\label{cHZprodproj}
c_{HZ}(\CP^{n_1}\times\dots \times\CP^{n_r}, a_1\omega_{FS}^1\oplus\cdots\oplus a_r\omega_{FS}^r)=(|a_1|+\cdots +|a_r|)\pi .
\end{equation} 
In fact,  Lu \cite{LU06}  was able to prove  that 
\begin{equation}\label{cHZprodineq}
c_{HZ}(\CP^{n_1}\times\dots \times \CP^{n_r}, a_1\omega^1_{FS}\oplus\cdots\oplus a_r\omega^r_{FS})\leq (|a_1|+\cdots +|a_r|)\pi
\end{equation}
which, combined with  (\ref{cHZlowerboundprodHSSCT}), yields (\ref{cHZprodproj}).
To  the authors' best knowledge no
 upper bound of  $c_{HZ}(M, \omega_{FS})$ is known for  HSSCT $(M, \omega_{FS})$, even for the case of  the complex Grassmannians (different from  the projective space).
 In Remark \ref{explanationnoupperbound} below we sketch the idea of   Lu's proof of the upper bound (\ref{cHZprodineq}) and  explain  why his argument cannot be used to achieve a similar upper bound for HSSCT.

\vskip 0.1cm

The following two  theorems summarize  our results  on Gromov width and Hofer--Zehnder capacity of Cartan domains.

\begin{theor}\label{main4}
Let $(\Omega, \omega_0)$ be a Cartan domain. Then 
\begin{equation}\label{cGcartan}
c_{G}(\Omega, \omega_0)=\pi
\end{equation}
and
\begin{equation}\label{cHZcartan}
c_{HZ}(\Omega, \omega_0)=\pi.
\end{equation}
Moreover,  if $\Omega_i\subset\C^{n_i}$, $i=1, \dots ,r$ are  Cartan domains of complex dimension $n_i$ equipped with the standard symplectic form $\omega_0^i$ of 
$\R^{2n_i}=\C^{n_i}$, then
\begin{equation}\label{cGproductCartan}
c_G \left( \Omega_1\times\dots\times   \Omega_r,  \omega_0^1\oplus\dots  \oplus \omega_o^r\right)=\pi.
\end{equation}
If  $a_1, \dots ,a_r$ are nonzero constants, then
\begin{equation}\label{cGupboundprodHSSNT}
c_G \left( \Omega_1\times\dots\times   \Omega_r, a_1 \omega_{0}^1\oplus\dots  \oplus a_r\omega_{0}^r\right)\leq \min \{|a_1|, \dots ,|a_r|\}\pi 
\end{equation}

\end{theor}

\begin{theor}\label{main5}
Let $(\Omega, \omega_0)$  be a  Cartan domain  and let  $(N, \omega)$ be   any closed symplectic manifold.
Then
\begin{equation}\label{cHZMprod}
c_{HZ}(N\times \Omega, \omega\oplus\omega_0)=\pi.
\end{equation}
\end{theor}

The proof of Theorem \ref{main4} which 
extends the results in \cite{LuDingQjao} valid  for classical Cartan domains  to  the product of  Cartan domains (including  the exceptional ones),
is based (together with the inclusion $B^{2n}(1)\subset (\Omega, \omega_0)$)  on the fact 
 that any $n$-dimensional  Cartan domain  $(\Omega, \omega_0)$ 
symplectically embeds into the cylinder $(Z^{2n}(1), \omega_0)$ (see Sections \ref{hpjts} and \ref{sectionembeddings}    for details).

The organization of the paper is as follows. In Section \ref{sectionLu} we summarize the above mentioned   Lu's work and some of  his results 
needed in this paper. In Section \ref{hpjts} we briefly  recall some  tools  on Hermitian positive Jordan triple systems which will be used in Section 
\ref{sectionembeddings} to construct the above mentioned embeddings of a Cartan domain into its compact dual, of the unit ball  into a  Cartan domain 
and of a Cartan domain into the unitary cylinder.
Moreover in Subsection \ref{sectionminatlas} we show how these symplectic embeddings could be used to get estimate and computation 
of the minimal number of Darboux charts needed to cover a HSSCT.
Finally, Section \ref{sectionproofs} is dedicated to the (conclusion of the) proofs of our theorems.
\vskip 0.3cm
\noindent {\bf Acknowledgments}. 
The authors  are indebted to Guy Roos for suggesting us the idea of  the construction of the symplectic  embedding of a Cartan domain into the unitary  cylinder  via the use of Hermitian positive Jordan triple systems. The authors would like also  to thank  Dietmar Salamon   for useful remarks on Gromov--Witten invariants.

\section{Lu's pseudo symplectic  capacities and Gromov--Witten invariants}\label{sectionLu}
G. Lu \cite{LU06} defines the concept of \emph{pseudo symplectic capacity} by weakening the  requirements for a symplectic capacity in such a way that this new concept depends on the homology classes of the symplectic manifold in question (the reader is referred to \cite{LU06} for more details.). More precisely, if one 
denotes by ${{\mathcal C} (2n, k)}$ the set of all tuples $(M, \omega; \alpha_1, \dots, \alpha_k)$ 
consisting of a $2n$-dimensional connected symplectic manifold $(M, \omega)$ and $k$ nonzero homology classes $\alpha_i\in H_*(M; \q )$, $i=1, \dots, k$,
a map $c^{(k)}$ from ${\mathcal C} (2n,k)$ to $[0, +\infty]$
is called a \emph{k-pseudo symplectic capacity}
 if it satisfies the following properties:

({\bf pseudo monotonicity}) if there exists a symplectic embedding $\varphi: (M, \omega_1)\rightarrow (M, \omega_2)$  then, for any $\alpha_i\in H_*(M_1;  \q)$,
$i=1, \dots k$,
\[c^{(k)}(M_1, \omega_1; \alpha_1,\dots ,\alpha_k)\leq c^{(k)}(M_2, \omega_2; \varphi_*(\alpha_1), \dots , \varphi_*(\alpha_k));\]  

({\bf conformality}) $c^{(k)}(M, \lambda\omega; \alpha_1, \dots, \alpha_k)=|\lambda|c^{(k)}(M, \omega; \alpha_1, \dots, \alpha_k)$, for every $\lambda\in\R\setminus \{0\}$
and all homology classes $\alpha_i\in H_*(M; \q)\setminus \{0\}$, $i=1, \dots ,k$;

({\bf nontriviality}) $c(B^{2n}(1), \omega_0; pt, \dots, pt)=\pi =c(Z^{2n}(1), \omega_0; pt, \dots, pt)$,
where we denote by $pt$ the homology class of a point.

\vskip 0.3cm

Note that if $k>1$ a $(k-1)$-pseudo symplectic capacity is defined by
\[c^{(k-1)}(M, \omega; \alpha_1, \dots , \alpha_{k-1}):=c^{(k)}(M, \omega;pt,  \alpha_1, \dots , \alpha_{k-1})\]
and any $c^{(k)}$ induces a true symplectic capacity 
\[c^{(0)}(M, \omega):=c^{(k)}(M, \omega; pt, \dots , pt).\]
Observe also that (unlike  symplectic capacities) pseudo symplectic capacities do not define symplectic invariants. 

\vskip 0.5cm

In   \cite{LU06} G. Lu  was able to construct two  $2$-pseudo symplectic capacities (christened by Lu   as \emph{pseudo symplectic capacities of Hofer--Zehnder type})
denoted by  $C_{HZ}^{(2)}(M, \omega; \alpha_1, \alpha_2)$ and $C_{HZ}^{(2o)}(M, \omega; \alpha_1, \alpha_{2})$ respectively (see Definition 1.3 and  Theorem 1.5 in  \cite{LU06}), where 
$\alpha_1$ and $\alpha_2$ are  homology classes\footnote{In the notations of \cite{LU06}  the generic classes $\alpha_1$  (resp. $\alpha_{2}$) are called $\alpha_0$ (resp. $\alpha_\infty$).
The reason for this notation comes from the concept  of hypersurface $S\subset M$ separating the homology classes $\alpha_0$ and $\alpha_{\infty}$  (see Definition 1.3 and the 
$(\alpha_0, \alpha_{\infty})$-Weinstein conjecture at p.6 of \cite{LU06}).} in $H_*(M;  \q)$.
Denote by 
\[C_{HZ}(M, \omega):=C_{HZ}^{(2)}(M, \omega; pt, pt)\]
 (resp. $C^{0}_{HZ}(M, \omega):=C_{HZ}^{(2o)}(M, \omega; pt, pt)$)
the corresponding  true symplectic capacities associated to Lu's pseudo symplectic capacities.
In the next lemma we summarize some   properties of the concepts involved so far.
\begin{lem}\label{lemmasumm}
Let $(M, \omega)$ be any symplectic manifold. 
Then, for arbitrary homology classes $ \alpha_1, \alpha_2\in H_*(M; \q)$   and for a nonzero homology class 
$\alpha$, with $\dim \alpha\leq \dim M-1$, the following inequalities hold true:
\begin{equation}\label{basiciC2}
C^{(2)}_{HZ}(M, \omega; \alpha_1, \alpha_2)\leq C^{(2o)}_{HZ}(M, \omega; \alpha_1, \alpha_{2})
\end{equation}
\begin{equation}
C^{(2)}_{HZ}(M, \omega; \alpha_1, \alpha_2)\leq C_{HZ}(M, \omega)\leq c_{HZ}(M, \omega)
\end{equation}
\begin{equation}
C_{HZ}^{(2o)}(M, \omega; \alpha_1, \alpha_2)\leq C^{o}_{HZ}(M, \omega)\leq c^{o}_{HZ}(M, \omega)
\end{equation}
\begin{equation}\label{cG<C2}
c_G(M, \omega)\leq C^{(2)}_{HZ}(M, \omega; pt, \alpha),
\end{equation}
where 
$c^{o}_{HZ}(M, \omega)$ is the $\pi_1$-sensitive Hofer--Zehnder capacity introduced in \cite{SCHWARZ00},
$c_{HZ}(M, \omega)$ is  the Hofer-Zehnder capacity  and $c_G(M, \omega)$ is the Gromov width of $(M, \omega)$.
Furthermore, if $M$ is closed then 
\[C_{HZ}(M, \omega)=c_{HZ}(M, \omega)\]
and
\[C_{HZ}^{o}(M, \omega)= c_{HZ}^{o}(M, \omega).\]
\end{lem}
\begin{proof}
See Lemma 1.4 and (12) in  \cite{LU06}.
\end{proof}

\begin{remark}\rm
It follows by (\ref{basiciC2}) and by the last two equalities
that for a closed symplectic manifold $(M ,\omega)$
\[c_{HZ}(M, \omega)\leq c^o_{HZ}(M, \omega).\]
Thus  
 inequality (\ref{cHZlowerboundprodHSSCT}) in Theorem \ref{main2}
holds  true  also when we replace $c_{HZ}$ with $c^o_{HZ}$.
\end{remark}

When the symplectic manifold is closed  the pseudo symplectic capacities   $C_{HZ}^{(2)}(M, \omega; \alpha_1, \alpha_2)$ and $C_{HZ}^{(2o)}(M, \omega; \alpha_1, \alpha_{2})$  can be estimated  by other two pseudo symplectic capacities $\GW(M, \omega; \alpha_1, \alpha_2)$ and \linebreak  $\GW_0(M, \omega; \alpha_1, \alpha_2)$ 
defined in terms of  \emph{Liu--Tian type Gromov-Witten invariants} as follows. Let $A\in H_2(M, \z)$:
the Liu--Tian  type  Gromov--Witten invariant of genus $g$ and with $k$ marked points  is a homomorphism
\[\Psi^M_{A, g, k}:H_*(\overline{{\mathcal M}}_{g, k}; \q)\times H_*(M; \q)^{k}\rightarrow \q , \ 2g+k \geq 3\]
where $\overline{{\mathcal M}}_{g, k}$ is the space of isomorphism classes of genus $g$ stable curves with $k$ marked points. When there is no risk of confusion, we will omit the superscript $M$ in $\Psi^M_{A, g, k}$.
Roughly speaking, one can think of $\Psi^M_{A, g, k}({\mathcal{C}}; \alpha_1, \dots, \alpha_k)$ as counting, for suitable generic $\omega$-tame almost complex structure $J$ on $M$, the number of $J$-holomorphic curves of genus $g$ representing $A$, with $k$ marked points $p_i$ which pass through cycles $X_i$ representing $\alpha_i$, and such that the image of the curve belongs to a cycle representing ${\mathcal{C}}$ (the reader is referred to the Appendix in \cite{LU06} and references therein for details).

\noindent In fact, there are several different constructions of Gromov-Witten invariants in the literature and the question whether they agree is not trivial (see \cite{LU06} and also Chapter 7 in \cite{MCSA94}). The most commonly used are the Gromov--Witten invariants described in the book of D. McDuff and S. Salamon \cite{MCSA94} 
which are homomorphisms  
\[\Psi_{A, g, m+2}: H_*(M; \q)^{m+2}\rightarrow \q ,\  m\geq 1\]
and which play an important role in the proofs of this paper. The Lemma \ref{lemmaugGW} below gives conditions under which these invariants agree with the ones considered by Lu.

\smallskip

\noindent Let $\alpha_1, \alpha_2\in H_*(M, \q)$.
Following \cite{LU06} one defines
\[\GW_g (M, \omega; \alpha_1, \alpha_2)\in (0, +\infty]\]
as the infimum of the $\omega$-areas $\omega (A)$ of the homology classes $A\in H_2(M, \z)$ for which the Liu--Tian  Gromov--Witten
invariant \linebreak $\Psi_{A, g, m+2}(C; \alpha_1, \alpha_2, \beta_1, \dots , \beta_m)\neq 0$
for some homology classes $\beta_1, \dots , \beta_m\in H_*(M, \q)$
and $C\in H_*(\overline{{\mathcal M}}_{g, m+2}; \q)$ and integer $m\geq 1$ (we use the convention $\inf \emptyset = + \infty$). The positivity of $GW_g$ reflects the fact that $\Psi_{A, g, m+2} = 0$ if $\omega(A) <0$ (see, for example, Section 7.5 in \cite{MCSA94}).
Set
\begin{equation}\label{GW}
\GW (M, \omega; \alpha_1, \alpha_2):=\inf \{\GW_g(M, \omega; \alpha_1, \alpha_2) \ | \ g\geq 0\}\in [0, +\infty].
\end{equation}

\begin{lem}\label{C2GW}
Let $(M, \omega)$ be a closed symplectic manifold. Then 
\[0\leq\GW (M, \omega;  \alpha_1, \alpha_2)\leq \GW_0 (M, \omega;  \alpha_1, \alpha_2).\]
Moreover $\GW (M, \omega;\alpha_1, \alpha_2)$ and $\GW _0(M, \omega; \alpha_1, \alpha_2)$
are pseudo symplectic capacities and, if the dimension $\dim M\geq 4$ then, for nonzero homology classes $\alpha_1, \alpha_2$,
we have
\[C_{HZ}^{(2)}(M ,\omega; \alpha_1, \alpha_2)\leq \GW (M, \omega; \alpha_1, \alpha_2)\]
\[C_{HZ}^{(2o)}(M ,\omega; \alpha_1, \alpha_2)\leq \GW_0 (M, \omega; \alpha_1, \alpha_2).\]
In particular, for every nonzero homology class $\alpha\in H_*(M ,\q)$,
\begin{equation}\label{CHZGW}
C_{HZ}^{(2)}(M ,\omega; pt, \alpha)\leq \GW (M, \omega; pt, \alpha)
\end{equation}
\begin{equation}\label{CHZGW0}
C_{HZ}^{(2o)}(M ,\omega; pt, \alpha)\leq \GW_0 (M, \omega; pt, \alpha).
\end{equation}
\end{lem}
\begin{proof}
See Theorems 1.10 and 1.13 in \cite{LU06}.
\end{proof}

\noindent We end this section with the following lemmata fundamental for the proof of our results.
Recall that a closed symplectic manifold is \emph{monotone} if there exists a number $\lambda >0$
such that $\omega (A)=\lambda c_1(A)$ for $A$ spherical (a homology class is called spherical if it is in the image of the Hurewicz homomorphism $\pi_2(M) \rightarrow H_2(M, \z)$). Further a homology class $A\in H_2(M, \z)$ is 
 \emph{indecomposable}  if it cannot be decomposed as  a sum $A=A_1+\cdots +A_k$, $k\geq 2$, of classes
which are spherical and satisfy $\omega (A_i)>0$ for $i=1, \dots , k$.
\begin{lem}\label{lemmaugGW} 
Let $(M, \omega)$ be a closed  monotone symplectic  manifold.
Let $A\in H_{2}(M, \z)$ be an indecomposable spherical  class, let $pt$ denote the class of a point in   $H_*(\overline{{\mathcal M}}_{g, m+2}; \q)$
and let   $\alpha_i\in H_*(M, \z)$, $i=1, 2, 3$.
Then the Liu--Tian Gromov--Witten invariant $\Psi_{A, 0, 3}(pt; \alpha_1, \alpha_2, \alpha_3)$ 
agrees with the Gromov--Witten invariant $\Psi_{A,0, 3}(\alpha_1, \alpha_2, \alpha_3)$.
\end{lem}
\begin{proof}
See \cite[Proposition 7.6]{LU06}.
\end{proof}

\begin{lem}\label{prodGromovWitten} 
Let $(N_1, \omega_1)$ and $(N_2, \omega_2)$
be two closed symplectic manifolds.
Then for every integer $k\geq 3$ and homology classes $A_2\in H_2(N_2; \z)$ and $\beta_i\in H_*(N_2; \z)$, $i=1, \dots, k$,
\[\Psi^{N_1\times N_2}_{0\oplus A_2, 0, k}(pt; [N_1]\otimes \beta_1,\dots , [N_1]\otimes \beta_{k-1}, pt\otimes \beta_k)=
\Psi^{N_2}_{A_2, 0, k}(pt; \beta_1, \dots , \beta_k).\]
\end{lem}
\begin{proof}
See  \cite[Proposition 7.4]{LU06}.
\end{proof}

\section{Hermitian positive Jordan triple system}\label{hpjts}

We refer the reader  to \cite{roos} (see also \cite{loos}) for more details on Hermitian symmetric spaces of noncompact type 
(HSSNT) and Hermitian positive Jordan triple systems (HPJTS).

\vskip 0.3cm

\noindent {\bf Definitions and notations.}
A Hermitian Jordan triple system is a  pair $\left({\mathcal M},
\{ ,  ,\}\right)$, where ${\mathcal M}$ is a complex vector space and $\{
,  ,\}$ is a map
\[
\{ ,  ,\}:{\mathcal M}\times {\mathcal M}\times {\mathcal M} \rightarrow {\mathcal M}
\]
\[
\left(u, \,  v, \,  w\right)\mapsto \{ u, \,  v, \,  w\}
\]
which is ${\C}$-bilinear and symmetric in $u$ and $w$, ${\C}$-antilinear in $v$ and such that the following \emph{ Jordan identity} holds:

\begin{equation*}\begin{split}
\left\{ x,  \, y, \,  \left\{ u, \,  v,  \, w\right\} \right\}  & -\left\{ u,  \, v, \,  \left\{ x,  \, y, \,  w\right\} \right\}= \\
& =\left\{ \left\{ x, \,  y, \,  u\right\} , \,  v,
 \, w\right\} -\left\{ u, \,  \left\{ v,  \, x,  \, y\right\} ,  \, w\right\}.
\end{split}\end{equation*}

For $x, \, y, \, z \in \M$ consider  the  operators
\[
T\left(x, \, y\right)z =\left\{  x, \, y, \, z\right\} 
\]
\[
Q\left(x, \, z\right) \, y =\left\{  x, \, y, \, z\right\}  
\]
\[
Q\left(x, \, x\right) =2\,Q\left(x\right)\label{D3}\\
\]
\[
B\left(x, \, y\right) =\operatorname{id}_{\mathcal M}-T\left(x, \, y\right)+Q\left(x\right)Q\left(y\right). \label{D4}
\]
The operator $B\left(x, \, y\right)$  is called the \emph {Bergman operator}. A Hermitian Jordan triple system is called \emph{positive} if the sesquilinear form 
\begin{equation}\label{D5}
\left(  u\mid v\right)  =\frac{1}{\gamma}\tr T\left(u, \, v\right)
\end{equation}
is a Hermitian product, where $\gamma$ is a positive constant called the \emph{genus} of 
$\left({\mathcal M},
\{ ,  ,\}\right)$.

\vskip 0.3cm

\noindent{\bf HSSNT associated to HPJTS.}
M. Koecher (\cite{Koecher1}, \cite{Koecher2}) discovered that to every HPJTS
$\left(\M, \{ ,  ,\}\right)$ one can associate an Hermitian symmetric
space of noncompact type, in its realization as circled\footnote{The domain $\Omega \subset \M$ is circled if $e^{i \theta}\cdot \Omega = \Omega$} bounded symmetric domain 
$\Omega_{\mathcal M}$ centered at the origin $0\in \M$. 
More precisely, $\Omega_{\mathcal M}$  is defined as the connected component containing the origin of   the set of all 
$u\in {\M}$ such that $B\left(u,  \, u\right)$ is positive definite with respect to the Hermitian product \eqref{D5}.

\vskip 0.3cm

\noindent{\bf HPJTS associated to HSSNT.} 
The HPJTS $\left({\M}, \{ ,  ,\}\right)$ can be recovered from its
associated HSSNT $\Omega_{\mathcal M}$  by defining ${\M}=T_0 \Omega_{\mathcal M}$ (the tangent space to the origin of $\Omega_{\mathcal M}$) and
\begin{equation*}\label{trcurvbis}
\{u,  \, v, \,  w\}=-\frac{1}{2} \, \left(R_0\left(u,  \, v\right) \, w+J_0 \, R_0\left(u, \,  J_0\,v\right)w\right),
\end{equation*}
where $R_0$ (resp. $J_0$) is the curvature tensor of the Bergman metric (resp. the complex structure) of $\Omega_{\mathcal M}$ evaluated at the origin.
The reader is referred   to  Proposition III.2.7 in  \cite{Bertram} for details. More informations on the correspondence between
HPJTS and HSSNT can also be found at p. 85 of Satake's book \cite{satake}.

\vskip 0.3cm

\noindent {\bf Spectral decomposition.}
Let $\left({\M}, \{ ,  ,\}\right)$ be a HPJTS. An element $c \in \M$ is called \emph {tripotent} if
$\{c, \, c, \, c\}=2 \, c$. Two tripotents $c_1$ and $c_2$ are called \emph {(strongly)
orthogonal} if $T\left(c_1,  \, c_2\right)=0$. Each element $v\in \M$ has a unique \emph{spectral decomposition}
\[
v=\lambda_{1} \, c_{1}+\cdots+\lambda_{s} \, c_{s}\qquad\left(\lambda_{1}>\cdots
>\lambda_{s}>0\right), \label{D7}
\]
where $\left(c_{1},\ldots, \, c_{s}\right)$ is a sequence of pairwise
orthogonal (with respect to (\ref{D5})) tripotents and the $\lambda_j$'s are real numbers called eigenvalues of $v$. The integer $s$  is  called the \emph{rank} of $v$ and is denoted by $\rk( v)$.
The \emph{rank} of $\M$  is the positive integer $r$  defined as  $r= \max \{\rk( z)\, |\, z \in \M    \}$. 
The elements $z\in\M$ such that $\operatorname{rk}(z)=r$ are called \emph{regular}.

Let us denote by   $\nax{v}$ the largest eigenvalue of $v$.
Due to the  convexity of $\Omega_{\mathcal M}$, $\nax{v}$ is  a norm on $\M$,
called the \emph{spectral norm}. 
The following proposition provides a description of  the domain $\Omega_{\mathcal M}$ in terms of its spectral norm.
\begin{prop}
Let $\Omega_{\mathcal M}\subset\M$ be the HSSNT associated to $\left({\M}, \{,  ,\}\right)$. Then 
\begin{equation}\label{Mball}
\Omega_{\mathcal M}=\{v \, |  \nax{v}<1\}.
\end{equation}
\end{prop}
\begin{proof}
See \cite[Corollary 3.15]{loos}.  
\end{proof}

\section{Cartan domains, their compact duals and some  symplectic embeddings}\label{sectionembeddings}
Let  $\left({\mathcal M},
\{ ,  ,\}\right)$ be a HPJTS and $\Omega_{\mathcal M}$ be its associated HSSNT.
Let $n$ be the complex dimension of $\M$.
By fixing an orthonormal basis
$\underline{e}=\{e_1, \dots , e_n\}$  of 
 $\left({\mathcal M},
\left(\cdot\mid \cdot\right) \right)$ 
we get the identification 
\begin{equation}\label{ident}
\M\rightarrow\C^{n},\  \  v\stackrel{\underline{e}}{\mapsto} z=(z_1, \dots, z_n),\ \  v=z_1e_1+\cdots  +z_ne_n ,
\end{equation}
which induces  an isometry 
between 
$(\M, \left(\cdot\mid \cdot\right)  )$
and $(\C^n, h_0)$, where $h_0$ is the canonical Hermitian product on $\C^n$.
Under the identification 
\[(z_1, \dots , z_n)=(x_1, y_1, \dots ,x_n, y_n)\] between $\C^n$ and $\R^{2n}$ 
we have  $h_0=g_0+i\omega_0$, where $g_0=\sum_{j=1}^ndx_j^2+dy_j^2$
is the standard  scalar product on $\R^{2n}$ and $\omega_0$ is the canonical symplectic form 
$\omega_0=\sum_{j=1}^ndx_j\wedge dy_j$
on $\C^n=\R^{2n}$.
From now on we assume $\M$ is  \emph{simple} which is equivalent
to the irreducibility of $\Omega_{\mathcal M}$.
Then, under the previous identification, 
the HSSNT $\Omega_{\mathcal M}$  corresponds  to a  bounded symmetric  domain $\Omega=\overline{e}(\Omega_{\mathcal M})\subset \C^n$.
The complex and Riemannian geometry of these domains, also called \emph{Cartan domains}, is well-known (see, e.g. \cite{Kobayashi}), 
Below, we describe some symplectic geometric aspects of these domains  
and their compact duals needed in this paper (for the concept of compact  dual see \cite{Helgason}
or  \cite{DiScalaLoi08} and references therein).

Let $\Omega\subset\C^n$ be a Cartan domain and let   $M$ be its compact dual. Then $M$ is an $n$-dimensional  HSSCT. Denote by
\begin{equation}\label{BW}
BW:M\rightarrow {\C}P^N
\end{equation}
 the \emph{Borel--Weil} (holomorphic)
embedding. It is well-known (see e.g. \cite{ta}) that the
pull-back $BW^*\omega_{\FS}$ of the Fubini--Study form
$\omega_{\FS}$ of ${\C}P^{N}$ is a homogeneous \K\--Einstein
form on $M$ ($\omega_{\FS}$ is the  \K\ form which, in the
homogeneous coordinates $[z_0,\dots, z_N]$ on ${\C}P^{N}$, is
given by $\omega_{\FS}=\frac{i}{2}
\partial\bar\partial\log (|z_0|^2+\cdots +|z_N|^2)$).
In this paper, we denote (with a
slight abuse of notation and terminology) by $\omega_{\FS}$ the
form $BW^*\omega_{\FS}$ and call it the {\em Fubini--Study
form} on $M$.  The symplectic form $\omega_{FS}$ can be equivalently described as the symmetric or canonical   form on $M$
normalized  so that $\omega_{FS} (A)=\pi$ for the generator  $A\in H_2(M, \z)$.

\vskip 0.3cm
\noindent
{\bf The domain $\left( \ncpt, \omega_0\right)$ can be embedded into $\left( M, \omega_\FS\right) $.}

\noindent
Let $(\Omega, \omega_0)$, $\Omega\subset\C^n$, be a Cartan domain equipped with the canonical symplectic  form $\omega_0$ of $\R^{2n}$
and  let $(M, \omega_{FS})$ be  its compact dual. 
In \cite{DiScalaLoi08} the first author in collaboration with A. J. Di Scala, by using HPJTS,  construct  an embedding
\begin{equation}\label{symplecticdual}
\Phi_{\Omega}:  \Omega   \f   M
\end{equation}
such that $\Phi_{\Omega}^*\omega_{FS}=\omega_0$.

Actually in \cite{DiScalaLoi08} much more is proved, namely that   
the embedding  $\Phi_{\Omega}$ induces a global  symplectomorphism
\[\Phi_{\Omega}: (\Omega, \omega_0)\rightarrow (M \setminus \Cut_0(M), \omega_{FS})\]
where $\Cut_0(M)$ is the cut locus of $(M, \omega_{FS})$ with respect to a fixed point $0\in M$ (see \cite[Theorem 1.1]{DiScalaLoi08}).
This diffeomorphism has been  christened in 
\cite{DiScalaLoi08}   as  a   \emph{symplectic duality} 
due to the fact that, amongst other properties,  it also satisfies 
$\Phi_{\Omega}^*\omega_0=\omega_{hyp}$, 
where $\omega_0$ denotes the standard form on $\C^n\cong M \setminus \Cut_0(M)$
and $\omega_{hyp}$ is  the hyperbolic metric on
 $\Omega$ (see either  \cite{DiScalaLoi08}  or \cite{DiScalaLoiRoos} for details  and  also \cite{CuccuLoi06},
  \cite{DiScalaLoiZuddas}, \cite{Calihom}, \cite{Riccisolitons}, \cite{diastexp},  \cite{LoiZuddas08} and \cite{mossa}  
  for  the construction of explicit  symplectic coordinates.

\begin{remark}\rm
In  \cite[Lemma 4.1 in Section 4]{LU06} it is shown the existence of a symplectic embedding
\begin{equation}\label{symplecticdualfirst}
\Phi_{\Omega_I[k, n]}: \Omega_I[k, n]\rightarrow G(k, n)
\end{equation}
from the first Cartan domain
$\Omega_I[k, n]\subset\C^{k(n-k)}$ into 
its compact  dual $G(k, n)$,  namely the   complex Grassmannian 
 of $k$ dimensional subspaces of $\C^n$.
Our result  (\ref{symplecticdual}) extends Lu's results to all HSSCT.
 \end{remark}

\noindent
{\bf The unitary   ball $(B^{2n}(1), \omega_0)$ can be   embedded into   $\left( \ncpt, \omega_0\right)$}.

\noindent
Let $v=\lambda_{1} \, c_1+\cdots+\lambda_{r}\,c_r$  be the spectral decomposition of a regular point $v\in \Omega_{\mathcal M}\subset\M$, then the distance $\di_0 (0,v)$ 
from the origin $0\in {\mathcal M}$ to $v$
 is given by 
\begin{equation}\label{flat distance}
\di_0( 0,v) = (v\mid v)^{\frac{1}{2}} = \sqrt{\sum_{j=1}^r\lmb^2_j},
\end{equation}
(see \cite[Proposition VI.3.6]{roos} for a proof).
Since  the set of regular points of $\M$ is dense (\cite[Proposition IV.3.1]{roos}) we conclude, by  \eqref{Mball}
and by the identification  $\Omega_{\mathcal M}\cong \Omega$ (induced by $(\M, \left(\cdot\mid \cdot\right))\cong(\C^n, h_0)$)
that  
\begin{equation}\label{ballintoomega}
(B^{2n}(1), \omega_0)\subset (\Omega, \omega_0).
\end{equation}
\endproof

 \begin{remark}\rm
The inclusion  (\ref{ballintoomega}) has been obtained in \cite[Lemma 4.2, Section 4]{LU06}
for  the case of the first Cartan domain, namely $B^{2k(n-k)}(1)\subset\Omega_I[k, n]$  (see also  \cite{LuDingQjao} for the case of  classical Cartan domains).
Combining this with the symplectic embedding (\ref{symplecticdualfirst}) 
Lu was able see \cite[Theorem 1.35]{LU06} to obtain the upper bound
\[F(G(k, n), \omega_{FS})\leq [n/k],\]
where $F(N, \omega)$ denotes the \emph{Fefferman invariant} of a closed symplectic manifold $(N, \omega)$, namely
the largest integer $p$ for which there exists a symplectic packing by $p$ open unit balls, and $[n/k]$ is  the largest integer less than or equal to $n/k$.
The authors believe it  is an intriguing problem (by using the techniques of this paper)
to give a similar  upper bound for all HSSCT.
\end{remark}

\noindent
{\bf  The domain $\left( \Omega, \omega_0\right)$ can be embedded into  $(Z^{2n}(1), \omega_0)$.}

\noindent
Let  $Z^{2n}(1)=\{(x, y)\ | \ x_1^2+y_1^2<1\}$ be the unitary cylinder in $\R^{2n}$.
Let $v=\lambda_{1} \, c_1+\cdots+\lambda_{r}\,c_r$ be the spectral decomposition of a regular point $v\in \Omega_{\mathcal M}\subset\M$. By \eqref{flat distance} and by the continuity of $\di _0$ (the distance function from the origin $0\in {\mathcal M}$) we see that $\di_0(0,\, c_1)=1$. Set $c:=c_1$, by \cite[Corollary 4.8]{loos} $c \in \de \Omega_{\mathcal M}$. As $\Omega_{\mathcal M}$ is convex (\cite[Corollary 4.7]{loos}), by the supporting hyperplane property there exists a real hyperplane $\pi$ of $\M$ through $c$ not intersecting 
$\Omega_{\mathcal M}$. Denote by   $p=\overline{e}(c)\in\de\Omega$ the  image of the  tripotent $c$ by  the isometry  (\ref{ident}).  Hence $p\in  S^{2n-1}$,
 where $S^{2n-1} = \de B^{2n}(1)$ is the  $\left( 2n-1\right) $-dimensional unit  sphere centered at the origin of $\R^{2n}$.
By (\ref{ballintoomega}), $B^{2n}(1) \subset \Omega=\overline{e}(\Omega_{\mathcal M})$ and hence  $\overline{e}(\pi)=T_{p}\, S^{2n-1}$. 
By applying the same argument to  any tripotent $c_{\theta}:=e^{i\theta}\cdot c$, we see that $\ncpt$ is contained  in the cylinder $\tilde Z$ bounded by the envelope of the family of real hyperplanes $\left\{  T_{p_{\theta}}\, S^{2n-1}, p_{\theta}=\overline{e}(c_{\theta})\right\}_{\theta \in\R }$. Let $W \in U(n)$ such that  
\begin{equation*}\label{assume}
W\cdot p  = \left( z_1,\, 0, \dots,\, 0\right)
\end{equation*}
\noindent for some $z_1 \in \C$, $\|z_1\|=1$.
It follows that  $W \cdot \tilde  Z = Z^{2n}(1)$ and 
the desired symplectic embedding of $(\Omega, \omega_0)$ into $(Z^{2n}(1), \omega_0)$ is
given by 
\begin{equation}\label{omegaintocil}
\Omega\subset\tilde Z\stackrel{W}{\rightarrow}Z^{2n}(1).
\end{equation}

\begin{remark}\rm
A similar (symplectic) embedding 
$(\Omega, \omega_0)\hookrightarrow (Z^{2n}(1), \omega_0)$
has been considered  in  \cite{LuDingQjao} for the classical Cartan domains.
\end{remark}

\subsection{Minimal symplectic  atlases of HSSCT}\label{sectionminatlas}
Consider a closed symplectic manifold $(M, \omega)$.
In  \cite{minatlas} Yu. B. Rudyak and F. Schlenk  have introduced  
the  \emph{symplectic Lustermik-Schnirelmann category} $S(M, \omega )$,
defined as
\[S(M, \omega)=\min \{k \ | \ M={\mathcal U}_1\cup\cdots \cup {\mathcal U}_k\}\]
where each ${\mathcal U}_i$ is the image $\Phi_i(U_i)$ of a symplectic embedding
$\Phi_i:U_i\rightarrow{\mathcal U}_i\subset M$ of a bounded subset $U_i$ of $(\R^{2n}, \omega_0)$
diffeomorphic to an open ball in $\R^{2n}$.
From our results  one obtains
the upper bound 
\begin{equation}\label{uppS}
S(M, \omega_{FS})\leq N+1
\end{equation}
for  Lustermik-Schnirelmann category of a Hermitian symmetric space of compact type $(M, \omega_{FS})$,
where $N$ is the dimension of the complex projective space $\C P^N$ where the manifold can be \K\ embedded 
via the  Borel--Weil embedding  $BW:M\rightarrow {\C}P^N$ (see (\ref{BW})).
Indeed, as in the case of the complex Grassmannian $G(k, n)$ (where  the Borel--Weil  embedding is given by the   Pl\"ucker embedding
$P:G(k, n)\rightarrow \CP^{\left({n\atop k}\right)-1}$),
one can define a \emph{ canonical atlas} on $(M, \omega_{FS})$ using the   $N+1$ holomorphic charts $\Omega_0, \dots , \Omega_{N}$ defined as   $\Omega_j=M\setminus\{BW^{-1}(Z_j=0)\}$,  and $Z_j=0$, $j=0, \dots, N$, is the standard hyperplane of $\CP^N$. Each $\Omega_j\subset\C^n$, $j=0, \dots , N$, is  biholomorphic to the noncompact dual $\Omega$ of $M$.  It follows by  (\ref{symplecticdual}) that 
$(\Omega_j, \omega_0)$ can be symplectically embedded into $(M, \omega_{FS})$ for $j=1, \dots ,N$.
On the other hand, each $\Omega$ is  a bounded domain  diffeomorphic to the ball in $\R^{2n}$
and so  (\ref{uppS}) follows. 
Our knowledge of the Gromov width of any HSSCT $(M ,\omega_{FS})$ can be used to estimate and compute the  minimal numbers of   Darboux charts needed to cover 
$M$. This number, introduced in \cite{minatlas}  and denoted there by $S_B(M, \omega)$, has been computed and estimated for various symplectic manifolds including  the complex Grassmannian (see \cite[Corollary 5.10]{minatlas}).
Similar computations and related problems (which will appear in a forthcoming paper)  can be done for all HSSCT using the results of this section.

\section{The proofs of Theorems \ref{main},  \ref{main2}, \ref{main3}, \ref{main4} and \ref{main5}}\label{sectionproofs}
The following lemma is the  key ingredient to achieve   the upper bound  of Gromov width in 
Theorems \ref{main},  \ref{main2} and \ref{main3}.

\begin{lem} \label{gromovwittenHSSCT}
Let $(M, \omega_{FS})$ be an irreducible HSSCT of complex dimension $n$ and let $A=[\CP^1]$
be a generator of $H_2(M, \z)$ such that $\omega_{FS}(A)=\pi$.
Then there exist $\alpha  {(M, \omega_{FS})}$ and $\beta (M, \omega_{FS})$ in $H_*(M, \z)$ such that
\[\dim \alpha  {(M, \omega_{FS})} + \dim\beta  {(M, \omega_{FS})}=4n-2c_1(A)\]
and
\begin{equation}\label{nonzeroGW}
\Psi_{A, 0, 3}(pt; \alpha  {(M, \omega_{FS})},  \beta  {(M, \omega_{FS})}, pt)\neq 0 .
\end{equation}
\end{lem}
\begin{proof}

Since the canonical symplectic form $\omega_{FS}$ is \K-Einstein, it follows that $(M, \omega_
{FS})$ is monotone, so that Lemma \ref{lemmaugGW}  applies under our assumptions. We need then to show the existence, for every irreducible HSSCT, of a non-vanishing Gromov-Witten invariant $\Psi_{A,0,3}(\alpha
(M, \omega_{FS}), \beta(M, \omega_{FS}), pt)$. This follows from the results about the {\it quantum cohomology} of these spaces proved in \cite{quantumquadric}, \cite{quantumexc}, \cite{quantumLagGrass}, \cite{quantumDIII}, \cite{SIEBTIAN}. 
Let us recall that the quantum cohomology ring of $M$ is the product $H_*(M) \otimes \z[q]$ endowed with the {\it quantum cup product}, defined for any two homology classes $\alpha, \beta \in H_*(M)$ as

\begin{equation}\label{qprod}
\alpha * \beta = \sum_{\gamma, d} \Psi_{dA,0,3}(\alpha, \beta, \gamma) \gamma^* q^d,
\end{equation}

\noindent the sum running over $d \in \z$ and $\gamma$ such that $\dim(\alpha) + \dim(\beta) + \dim(\gamma) = 4n - 2 dc_1(A)$, 
where $\gamma^*$ denotes the dual class of $\gamma$. 

\noindent Looking at the formulas for the quantum product proved in the above-mentioned references, it is not hard to find a Gromov-Witten invariant $\Psi_{A,0,3}(\alpha, \beta, pt)$ which does not vanish for some classes $\alpha$, $\beta$. More in detail, when $M$ is the Grassmannian $G(k,n)$, by \cite{SIEBTIAN} there exist $\alpha \in H_{2k(n-1)}(M)$ and $\beta \in H_{2n(k-1)}(M)$ such that this holds; by \cite{quantumDIII} the same is true for suitable $\alpha = \beta \in H_{(n-1)(n-2)}(SO(2n)/U(n))$; by Corollary 8 in \cite{quantumLagGrass} $\alpha$ and $\beta$ can be taken of codimension $n$ and $1$ when $M$ is the Lagrangian Grassmannian $LG(n, 2n)$; in 
\cite{quantumexc} (see the formulas in Sections 5.1 and 5.2) it is shown that for the Cayley plane (resp. for the Freudenthal variety) one can take for example $\alpha$ and $\beta$ of codimensions 8 and 4, (resp. of codimensions 13 and 5). Finally, in \cite{quantumquadric} is studied the quantum cohomology of complete intersections, which in particular gives a non-vanishing Gromov-Witten invariant for the complex quadric.
\end{proof}

\begin{remark}\rm
\noindent Formulas for quantum products in the homogeneous spaces, expressed in terms of the combinatorial invariants of the Lie algebra of the symmetry group of the space (Dynkin diagram and Weyl group), can be found in \cite{FULTON04bis} (see also \cite{FULTON97}, \cite{FULTON04}) and could be also used to prove the above Lemma.
\end{remark}
We are now in the position to prove Theorem \ref{main}. 
\begin{proof}[Proof of Theorem \ref{main}]
In order to use  Lemma \ref{C2GW} we can assume, without loss of generality, that  $\dim M\geq 4$.
Indeed the only irreducible HSSCT of dimension $\leq 4$
are either  $(\CP^1, \omega_{FS})$ or  $(\CP^2, \omega_{FS})$
whose  Gromov width is well-known to be equal to $\pi$.
Let $A=[\C P^1]$ be the generator of $H_2(M, \z)$ as in the statement of Theorem \ref{main}. Then the value $\omega_{FS} (A)=\pi$
is clearly the infimum of the $\omega_{FS}$-areas $\omega_{FS} (B)$ of the homology classes $B\in H_2(M, \z)$ for which $\omega_{FS}(B) > 0$.

\noindent By Lemma \ref{gromovwittenHSSCT} we have $\Psi_{A, 0, 3}(pt; pt, \alpha , \beta ) \neq 0,$
with $\alpha=\alpha (M, \omega_{FS})$ and  $\beta=\beta (M, \omega_{FS})$, and hence, by definition of $GW_g$,
\begin{equation}\label{GW=GW0}
\GW (M, \omega_{FS}; pt, \gamma)=\GW_0 (M, \omega_{FS}; pt, \gamma) =\pi
\end{equation}
with $\gamma=\alpha (M, \omega_{FS})$ or  $\gamma=\beta (M, \omega_{FS})$.
It follows by the  inequalities
(\ref{basiciC2}), (\ref{cG<C2}), (\ref{CHZGW}) and   (\ref{CHZGW0}) that 
\begin{equation}\label{upperboundcG}
c_G(M, \omega_{FS})\leq C_{HZ}^{(2)}(M ,\omega_{FS}; pt, \gamma) \leq
C_{HZ}^{(2o)}(M ,\omega_{FS}; pt, \gamma)\leq\pi 
\end{equation}
with $\gamma=\alpha (M, \omega_{FS})$ or  $\gamma=\beta (M, \omega_{FS})$.
Combining this with the lower bound $c_G(M, \omega_{FS})\geq\pi$ coming from the inclusion $B^{2n}(1)\subset(\Omega, \omega_0)$ (cfr. (\ref{ballintoomega})), the symplectic embedding $\Phi_{\Omega}:(\Omega, \omega_0)\rightarrow (M, \omega_{FS})$ (cfr. (\ref{symplecticdual}))
and the monotonicity and nontriviality of $c_G$,  one gets:
\begin{equation}\label{variousequalities}
c_G(M, \omega_{FS})= C_{HZ}^{(2)}(M ,\omega_{FS}; pt, \gamma)=
C_{HZ}^{(2o)}(M ,\omega_{FS}; pt, \gamma)=\pi
\end{equation}
with $\gamma=\alpha (M, \omega_{FS})$ or  $\gamma=\beta (M, \omega_{FS})$.
This concludes the proof of Theorem \ref{main}.
\end{proof}

\begin{remark}\rm
Observe that we have  proven  more than stated in Theorem \ref{main}. Indeed,  we have computed   the value of
Lu's pseudo symplectic capacities evaluated  at the homology class of a point and at $\alpha  {(M, \omega_{FS})}$ (or $\beta {(M, \omega_{FS})}$), namely
\[c_G(M, \omega)= C_{HZ}^{(2)}(M ,\omega; pt, \alpha  {(M, \omega_{FS})})=C_{HZ}^{(2o)}(M ,\omega; pt, \alpha  {(M, \omega_{FS})})=\]
\ \ \ \ \ \ \ \ \ \ \ \ \ \ \ $=C_{HZ}^{(2)}(M ,\omega; pt, \beta {(M, \omega_{FS})})=C_{HZ}^{(2o)}(M ,\omega; pt, \beta  {(M, \omega_{FS})})=\pi.$

\noindent
This extends the result obtained by G. Lu for the  complex Grassmannian
(cfr.  \cite[Theorem 1.15]{LU06} for details) 
 to HSSCT.
\end{remark}

\begin{remark}\rm
An alternative proof of  the upper bound $c_G(M, \omega_{FS})\leq\pi$ in Theorem \ref{main} can be achieved by combining  Lemma \ref{gromovwittenHSSCT} with \cite[Proposition 4.1 ]{GWgrass}
which asserts  that  if $(M, \omega)$ is a symplectic manifold of (real) dimension $2n$, $A\in H_2(M, \z)$ is an indecomposable spherical class 
and $\Phi_{A, 0, 3}(pt, \alpha_0, \beta_0)\neq 0$, for suitable $\alpha_0$ and $\beta_0$ in $H_*(M, \z)$ (which necessarily satisfy $\dim \alpha_0 +\dim \beta_0=4n-2c_1(A)$) then 
$c_G(M, \omega)\leq \omega (A)$.

 Using the same idea, A. C. Castro (\cite{castro}) recently found an upper bound for the Gromov width of homogeneous spaces $G/K$, with $G = SU(n)$, endowed with a general $SU(n)$-invariant symplectic structure $\omega$, provided a technical assumption on $\omega$ aimed to assure that some homology classes are $\omega$-indecomposable.
When $G/K$ is the Grassmannian manifold, endowed with the \K-Einstein structure normalized as in our paper, we recover our result.
The same method was used also in \cite{GWcoadjoint} to bound from above the Gromov width of $G/K$ when $G$ is any simple compact group and $K$ a maximal torus in $G$ (notice that the only HSSCT satisfying these assumptions is $\CP^1$).

 In \cite{GWgrass} the exact value of the Gromov width of the Grassmannian manifold is in fact calculated by giving also the lower bound. This follows from the existence of Hamiltonian circle actions on the manifold with an isolated fixed point and isotropy weights equal to 1, which by a refined version of the equivariant Darboux theorem allows the author to find an open set of the manifold equivariantly symplectomorphic to the Euclidean ball.
This approach could probably be adapted to prove the lower bound at least for the quotients of the classical groups.
\end{remark}

In order to prove Theorem \ref{main2} we need the following lemma, interesting on its own sake,  which extends Lu's formula (20) in  \cite[Theorem 1.16]{LU06} (for the Grassmannian) to the case of HSSCT.

\begin{lem}
Let $(M, \omega_{FS})$ be a HSSCT and let $(N, \omega)$ be any closed symplectic manifold.
Then 
\begin{equation}\label{C2oNM}
C_{HZ}^{(2o)}(N\times M, \omega\oplus a\omega_{FS}; pt, [N]\times\gamma)\leq |a|\pi
\end{equation}
for any $a\in\R\setminus\{0\}$ and $\gamma=\alpha  {(M, \omega_{FS})}$ or $\gamma=\beta  {(M, \omega_{FS})}$, 
with $\alpha  {(M, \omega_{FS})}$ and $\beta  {(M, \omega_{FS})}$ given by Lemma \ref{gromovwittenHSSCT}.
\end{lem}
\begin{proof}
Since by   (\ref{nonzeroGW}) we have $\Psi^M_{A, 0, 3}(pt; \alpha,  \beta , pt))\neq 0,$
with $\alpha =\alpha (M, \omega_{FS})$ and  $\beta =\beta (M, \omega_{FS})$,
it follows by Lemma  \ref{prodGromovWitten}  that
\[\Psi^{N \times M}_{B, 0, 3}(pt; [N]\times \alpha (M, \omega_{FS}), [N]\times \beta (M, \omega_{FS}), pt )\neq 0
\]
for $B=0\times A$, where $0$ denotes the zero class in $H_2(N, \z)$ and $A$ the generator of $H_2(M, \z)$.
Hence (\ref{C2oNM})  easily follows   from (\ref{CHZGW0}) in Lemma \ref{C2GW}.
\end{proof}

\begin{proof}[Proof of Theorem \ref{main2}]
To see (\ref{cGprodHSSCT}) we assume $r>1$ because of the result in Theorem \ref{main}.
It immediately follows  from  (\ref{basiciC2}) and  (\ref{cG<C2}) in Lemma \ref{lemmasumm} and by (\ref{C2oNM}) that 
\[c_G \left( M_1\times\dots\times   M_r,  \omega_{FS}^1\oplus\dots  \oplus \omega_{FS}^r\right)\leq\pi.\]
On the other hand, 
we have the symplectic embeddings 
\[\times_{j=1}^rB^{2n_j}(1)\subset \times_{j=1}^r \Omega_j
\stackrel{\Phi_{\Omega_1}\times\cdots\times\Phi_{\Omega_r}}{\longrightarrow}  \times_{j=1}^rM_j\]
(induced by (\ref{ballintoomega}) and (\ref{symplecticdual}) respectively)
and the natural inclusion
\begin{equation}\label{ballintoball}
B^{2n_1+\cdots +2n_r}(1)\subset \times_{j=1}^rB^{2n_j}(1).
\end{equation}
Thus,
it follows  by the monotonicity and nontriviality of $c_G$ that 
\[c_G \left( M_1\times\dots\times   M_r,  \omega_{FS}^1\oplus\dots  \oplus \omega_{FS}^r\right)\geq\pi .\]
Hence (\ref{cGprodHSSCT}) follows. As we have already pointed out in the Introduction, inequality  (\ref{cGupboundprodHSSCT})
is a straightforward consequence of   (\ref{cGMprodHSSCT}) in Theorem \ref{main3}.  

 Inequality (\ref{cHZlowerboundprodHSSCT}) follows by (\ref{cGHSSCT}), by the monotonicity of $c_{HZ}$
 and  from the fact that for 
 two compact  symplectic manifolds $(N_1, \omega_1)$ and $(N_2, \omega_2)$
 \begin{equation}\label{incHZ}
 c_{HZ}(N_1\times N_2, \omega_1\oplus \omega_2)\geq c_{HZ}(N_1, \omega_1)+c_{HZ}(N_2, \omega_2)
 \end{equation}
(see \cite[Lemma 4.3, p. 43]{LU06} for a proof).
 This concludes the proof of Theorem \ref{main2}.
\end{proof}

\begin{remark}\rm
The upper bound  
\[c_G \left( M_1\times\dots\times   M_r,  \omega_{FS}^1\oplus\dots  \oplus \omega_{FS}^r\right)\leq \pi\]
obtained in the proof of Theorem \ref{main2}
can also be achieved by using the fact that HSSCT and their products  are uniruled  manifolds (see Definition 1.14,  Theorem 1.27 in 
\cite{LU06} and the remark following this theorem).
\end{remark}

\begin{remark}\rm
Note that in \cite[Theorem 1.16]{LU06}  another interesting result is proven namely formula (21).
Using the techniques developed so far one can prove the analogous of this formula, namely
\[C_{HZ}^{(2o)}(\times_{j=1}^rM_j,\oplus_{j=1}^r a_j\omega^j_{FS}; pt, \times_{j=1}^r\alpha_j)\leq (|a_1|+\cdots +|a_r|)\pi,\]
for all $a_j\in\R\setminus\{0\}$ and $\alpha_j=\alpha_j(M_j, \omega^j_{FS})$ or $\beta_j=\beta_j(M_j, \omega^j_{FS})$.
\end{remark}

\begin{remark}\rm\label{explanationnoupperbound}
We do not know if  the inequality 
\[c_{HZ} \left( M_1\times\dots\times   M_r, a_1 \omega_{FS}^1\oplus\dots  \oplus a_r\omega_{FS}^r\right)\leq (|a_1|+ \dots +|a_r|)\pi.\]
holds true.
Unfortunately, the proof  given by Lu  in the case of product of projective spaces  and 
\cite[Theorem 1.21]{LU06}) do not extend to the general case of HSSCT.
Indeed the Gromov--Witten invariant \linebreak $\Psi_{A, 0, m+2}(pt, pt, \beta_1, \dots , \beta_m)$ of  
$M = M_1\times\dots\times   M_r$
does not vanish  (for some homology classes $\beta_1, \dots , \beta_m$)
 if and only if   all the  $M_j$'s are   projective spaces, since it is easily checked that the dimension condition $\sum_{j=1}^m \deg(\beta_j) = 2(c_1(A) - \dim(M) - 1 + m)$, necessary for the Gromov-Witten invariant to be nonzero (\cite{MCSA94}, p. 11), is satisfied only in this case. 
The reader is also referred to  \cite[Corollary 1.19 and Example 1.20]{LU06})
for some comments and conjectures related to this problem.
 \end{remark}

\begin{proof}[Proof of Theorem \ref{main3}]
It follows from  (\ref{basiciC2}) and  (\ref{cG<C2}) in Lemma \ref{lemmasumm}  and by (\ref{C2oNM}) that 
\[c_G(N\times M, \omega\oplus a\omega_{FS})\leq  C_{HZ}^{(2o)}(N\times M, \omega\oplus a\omega_{FS}; pt, [N]\times\gamma)\leq |a|\pi,\]
where $\gamma=\alpha  {(M, \omega_{FS})}$ (or $\gamma=\beta  {(M, \omega_{FS})}$),
yielding the desired inequality (\ref{cGMprodHSSCT}).
\end{proof}

\begin{proof}[Proof of Theorem \ref{main4}]
 By 
\[(B^{2n}(1),\omega_0)\subset  (\Omega, \omega_0)\stackrel{W}{\rightarrow} (Z^{2n}(1), \omega_0),\]
(given by (\ref{ballintoomega})  and (\ref{omegaintocil}) respectively) and  the monotonicity and nontriviality of $c_{G}$ and $c_{HZ}$
we get $c_G(\Omega, \omega_0) =c_{HZ}(\Omega, \omega_0)= \pi$, namely
 (\ref{cGcartan}) and (\ref{cHZcartan}).
Analogously, let us denote $M_j$ the compact dual of $\Omega_j$: by (\ref{cGupboundprodHSSCT}) and by the symplectic  embedding
\[(\times_{j=1}^r\Omega_j, \oplus_{j=1}^ra_j \omega_{0}^j)\stackrel{\Phi_{\Omega_1}\times\cdots\times\Phi_{\Omega_r}}{\longrightarrow} (\times_{j=1}^rM_j, \oplus_{j=1}^ra_j \omega_{FS}^j)\]
induced by (\ref{symplecticdual})
one obtains  (\ref{cGupboundprodHSSNT}) which, together with the symplectic  embedding $\times_{j=1}^rB^{2n_j}(1)\subset \times_{j=1}^r\Omega_j$
 (induced by (\ref{ballintoomega})) 
and (\ref{ballintoball}) yields  (\ref{cGproductCartan}).
\end{proof}
In order to prove Theorem \ref{main5} we need the following  interesting  result of Lu.
\begin{lem}\label{lemmaludeep}
Let $(N ,\omega)$ be any closed  symplectic manifold.
Then, for any $r>0$ one has
\[c_{HZ}(N\times B^{2n}(r), \omega\oplus\omega_0)=c_{HZ}(N\times Z^{2n}(r), \omega\oplus\omega_0)=\pi r^2.\]
where $Z^{2n}(r)$ is given by (\ref{Zcil}).
\end{lem}
\begin{proof}
See \cite[Theorem 1.17, p.14]{LU06}.
\end{proof}
\begin{proof}[Proof of Theorem \ref{main5}]
By  $(B^{2n}(1),\omega_0)\subset  (\Omega, \omega_0)\stackrel{W}{\rightarrow} (Z^{2n}(1), \omega_0)$ 
one has the embeddings 
\[(N\times B^{2n}(1), \omega\oplus\omega_0)\subset (N\times \Omega, \omega\oplus\omega_0)\stackrel{id_N\times W}{\rightarrow} (N\times Z^{2n}(1), \omega\oplus\omega_0)\] 
 and so the desired   (\ref{cHZMprod}), i.e.
$c_{HZ}(N\times \Omega, \omega\oplus\omega_0)=\pi$,
follows by  Lemma \ref{lemmaludeep} and the monotonicity of $c_{HZ}$.
\end{proof}

\noindent
{\bf Final remarks on Seshadri constants}

\noindent
Our knowledge of the Gromov width  of  a HSSCT  allows us to obtain an upper bound of the Seshadri constant 
of an ample  line bundle   over a HSSCT $(M, \omega_{FS})$.
Recall that given a compact complex manifold $(N, J)$
and an holomorphic line bundle  $L\rightarrow N$
the \emph{Seshadri constant} of $L$ at a point $x\in N$ is defined as the nonnegative real number 
\[\epsilon (L, x)=\inf_{C\ni x}\frac{\int_Cc_1(L)}{\mult_xC}, \]
where the infimum is taken over all irreducible holomorphic curves $C$ passing through the point $x$
and $\mult_xC$ is the multiplicity of $C$ at $x$ (see \cite{DE} for details). 
The (global) Seshadri constant is defined by
\[\epsilon (L)=\inf_{x\in M}\epsilon (L, x).\]
Note that Seshadri's criterion for ampleness says that  $L$ is ample if and only if  $\epsilon (L)>0$.
P. Biran and K. Cieliebak \cite[Prop. 6.2.1]{BICI01} have  shown  that  
\[\epsilon (L)\leq c_G(M, \omega_L),\]
where $\omega_L$ is any  \K\ form 
which represents the first Chern class of $L$, i.e. $c_1(L)=[\omega_L]$.
Consider now  an irreducible  HSSCT $(M, \omega_{FS})$ and 
the line bundle  $L\rightarrow M$  such that 
$c_1(L)=[\frac{\omega_{FS}}{\pi}]$
($L$ can be taken as the pull-back via the Borel--Weil embedding (\ref{BW}) of the universal bundle of $\CP^N$).
Therefore, by using the upper bound $c_G(M, \omega_{FS})\leq\pi$ and the conformality of $c_G$ 
we get: 
\begin{cor}
Let $(M, \omega_{FS})$ be an irreducible HSSCT and let $L\rightarrow M$
as above. Then $\epsilon (L)\leq 1.$
\end{cor}

\end{document}